\documentclass{amsart}
\usepackage{graphicx} 
\usepackage{amsfonts,amsmath,tikz,caption}
\usepackage{amssymb}
\usepackage{pb-diagram}
\usepackage[all]{xy}
\usepackage[hidelinks]{hyperref}
\usepackage{xcolor}
\usetikzlibrary{decorations.markings}
\usetikzlibrary{shapes}
\usetikzlibrary{backgrounds}
\usepackage{subfig}
\usepackage{tikz-cd}

\newtheorem{thm}{Theorem}

\newtheorem{lemma}[thm]{Lemma}
\newtheorem{conj}[thm]{Conjecture}

\newtheorem{quest}[thm]{Question}

\theoremstyle{definition}

\theoremstyle{remark}
\newtheorem{rmk}[thm]{Remark}

\numberwithin{thm}{section}

\newcommand{\Z}{\mathbb{Z}}
\newcommand{\N}{\mathbb{N}}

\newcommand{\C}{\mathbb{C}}

\newcommand{\D}[1]{\mathcal{D}_{#1}}

\DeclareMathOperator{\homol}{H}

\newcounter{samcomments}

\newcounter{yashcomments}

\title{A note on normal generation and the first $\ell^2$-Betti number}
\begin{document}

\author{Sam P.~Fisher}
\address{Instituto de Ciencias Matem\'aticas, CSIC-UAM-UC3M-UCM, Madrid, Spain}
\email[S.~P.~Fisher]{samuel.fisher@icmat.es}

\author{Yash Lodha}
\address{Department of Mathematics\\ Purdue University\\ West Lafayette, Indiana, USA}
\email[Y.~Lodha]{ylodha@purdue.edu}

\date{\today}

\begin{abstract}
In $2011$, Osin and Thom conjectured that the first $\ell^2$-Betti number of a torsion-free discrete group is bounded above by the normal rank of the group minus one.
The conjecture has surprising consequences for some fundamental problems in group theory and topology. These include the Wiegold problem on perfect groups, the Levin conjecture, the torsion-free case of the Kervaire conjecture, and an important special case of the Whitehead asphericity conjecture.
In this article, we construct for each $n\in \N$ a countable torsion-free group $\Gamma_n$ such that $\beta^{(2)}_1(\Gamma_n)=n$ and so that the normal rank, $n(\Gamma_n)$, equals one. This disproves the conjecture. Our counterexamples are locally free and hence locally indicable. However, they are not finitely generated.
\end{abstract}

\maketitle

\section{Introduction}

The rank of a group $G$, or $d(G)$, is the smallest cardinality of a generating set.
The normal rank, or $n(G)$, is the smallest size of a subset of $G$ whose normal closure is $G$.
The normal rank is a classical notion that remains highly mysterious and especially hard to estimate from below.
The inequality $d(G/[G,G])\leqslant n(G)$ provides a natural lower bound, and few results have provided better lower bounds in the case when $G$ is perfect.
Motivated by the inequality $\beta^{(2)}_1(G)\leqslant d(G)-1$ for finitely generated groups, Osin and Thom conjectured the following \cite[Conjecture 1.3]{OsinThom2013}:

\begin{conj}[Osin--Thom]\label{ConjOsinThom}
If $G$ is torsion-free and discrete, then $\beta^{(2)}_1(G)\leqslant n(G)-1$.
\end{conj}

In the same article, Osin and Thom established the conjecture for limits of amenable left orderable groups.
In addition to a possible deep connection between $\ell^{2}$-cohomology and normal rank, the conjecture has some surprising consequences for several central problems in group theory and topology.
Recall that for an infinite group $G$ the following holds \cite{Luck2002}: $$\beta_1^{(2)}(G\star G)=2\beta_1^{(2)}(G)+1\qquad \beta_1^{(2)}(G\star \Z)=\beta_1^{(2)}(G)+1.$$
So if Conjecture \ref{ConjOsinThom} were true, it would imply that: $$n(G\star G)\geqslant \beta_1^{(2)}(G\star G)+1\geqslant 2.$$ for any torsion-free discrete group $G$. The Levin conjecture \cite{Levin1962} is closely related to the statement that $n(G\star G)\geqslant 2$ for $G$ torsion-free, which remains a longstanding open problem \cite{ChenLodha2025}. 

Moreover, if true, Conjecture \ref{ConjOsinThom} provides a solution to the Wiegold problem on perfect groups which asks if there exists a finitely generated perfect group $G$ with $n(G)>1$ \cite{ChenLodha2025}. A solution follows from Conjecture \ref{ConjOsinThom} upon choosing such a $G$ to be nontrivial, finitely generated, perfect, and torsion-free. The Wiegold problem was recently solved by the second author and Chen \cite{ChenLodha2025}, who proved that $n(G\star G)>1$ when $G$ is left orderable. 

The Kervaire conjecture asserts that if $A$ is a nontrivial group, then $n(A\star \Z)>1$. This has been settled for the class of torsion-free groups by Klyachko \cite{Klyachko1993} (see \cite{Chen2026} for a recent new proof). Notably, Conjecture \ref{ConjOsinThom} supplies an alternative proof of this result of Klyachko.
The Whitehead asphericity conjecture asserts that given any connected and aspherical $2$-complex $X$, any connected subcomplex $Y$ of $X$ is also aspherical.
Howie showed that the compact case of the Whitehead asphericity conjecture is equivalent to the following statement \cite{Howie1983}.
Let $X=Y\cup_fD^2$ where $D^2$ is a $2$-cell attached along $f\colon \partial D^2\to Y^{(1)}$ to obtain $X$. Then the conjecture asserts that if $X$ is contractible, then $Y$ is aspherical. 
Indeed, in this case it follows that $n(\pi_1(Y))=1$ since attaching the $2$-cell $D$ makes it simply connected. Strikingly, Conjecture \ref{ConjOsinThom} provides a positive solution in this case when $\pi_1(Y)$ is torsion-free, upon applying a result of Berrick and Hillman \cite{BerrickHillman2008} who showed that $\beta^{(2)}_1(\pi_1(Y))=0$ implies that such a complex $Y$ is aspherical. 

In this article, we disprove the conjecture and show the following.

\begin{thm}\label{thm:main-thm}
For each $n\in \N$, there is a countable torsion-free group $\Gamma_n$ such that $\beta^{(2)}_1(\Gamma_n)=n$ and $n(\Gamma_n)=1$. Moreover, each $\Gamma_n$ is locally free (and hence locally indicable).
\end{thm}

The construction is as follows.
Fix an integer $n\in \N$. For each $m\in \N$, let $F_{n+1}^{(m)}$ be the free group freely generated by $\{a_m,s_{m,1},\ldots, s_{m,n}\}$.
We define an injective homomorphism $\phi_m\colon F_{n+1}^{(m)}\to F_{n+1}^{(m+1)}$ by: 
\[a_m\mapsto a_{m+1}\qquad s_{m,i}\mapsto s_{m+1,i}^{-1} a_{m+1} s_{m+1,i}.\] 
This map is injective since the kernel of the retraction from $F_{n+1}^{(m+1)}$ to the free subgroup $K_n=\langle s_{m+1,1},\ldots, s_{m+1,n}\rangle$
is precisely the free product of the groups $\{w^{-1} \langle a_{m+1}\rangle w\mid w\in K_n\}$.
The group $\Gamma_n$ is then defined as the directed union $\bigcup_{m\in \N} F_{n+1}^{(m)}$.
Let $a\in \Gamma_n$ be the image of each $a_m\in F_{n+1}^{(m)}$. It follows immediately from the construction that $a$ normally generates $\Gamma_n$.
It is also immediate that $\Gamma_n$ is locally free and hence locally indicable.
In this article, we prove that $\beta_1^{(2)}(\Gamma_n)=n$.
Since our groups are not finitely generated, this leaves open the following question:

\begin{quest}
Does the Osin--Thom conjecture hold for finitely generated torsion-free groups?
\end{quest}

The authors wish to emphasize that the addition of the finite generation hypothesis in the Osin--Thom conjecture makes the problem much more interesting. In particular, the consequences towards the group theory and topology problems described above essentially already follow from the finitely generated version. Indeed, the authors believe that the reformulated problem is an important and difficult open problem which deserves the attention of researchers in the area.

\subsection{Acknowledgments}

The first author is supported by the grants CEX2023-001347-S and EUR2025-164928 of the Ministry of Science, Innovation, and Universities of Spain.
The second author is supported by the NSF CAREER award DMS-2552707. The construction and the proofs are due to the authors, but the second author acknowledges the use of ChatGPT to aid in developing background and intuition for the calculations of $\ell^2$-Betti numbers which were helpful.

\section{Preliminaries}

We give a very brief review of group homology, and refer the reader to \cite{Brown1982} for more details. Let $G$ be a group, and consider a free resolution of the trivial $\C[G]$-module $\C$:
\[\dots\rightarrow P_2\rightarrow P_1\rightarrow P_0\rightarrow \C\rightarrow 0.\]
If $M$ is a right $\C[G]$ module, tensoring the resolution with $M$ produces the following chain complex:
\[\dots \rightarrow M\otimes_{\C[G]}P_2\rightarrow M\otimes_{\C[G]}P_1\rightarrow M\otimes_{\C[G]}P_0\to 0.\]
The degree $n$ homology of this chain complex is the degree $n$ \emph{homology of $G$ with coefficients in $M$}, and denoted by $\homol_n(G;M)$.
Up to isomorphism, the homology is independent of the choice of free resolution.

We introduce a convenient choice of a free resolution in order to prove Lemma \ref{lem:direct-union-homology} below; this is the so-called \emph{bar resolution}.
For each
$n\geqslant 0$, let $B_n(G)$ be the free left $\mathbb{C}[G]$-module with
basis
\[
    \bigl\{[h_1\mid \cdots \mid h_n]\mid h_1,\ldots,h_n\in G\bigr\}.
\]
For $n=0$, the unique basis element is simply denoted by $[\ ]$, so that
$B_0(G)\cong \mathbb{C}[G]$. The boundary map $\partial_n \colon B_n(G) \to B_{n-1}(G)$ for $n \geqslant 1$ is defined by
\[
\begin{aligned}
\partial_n [h_1\mid\cdots\mid h_n]
={}&h_1[h_2\mid\cdots\mid h_n]\\
&+\sum_{j=1}^{n-1}(-1)^j
 [h_1\mid\cdots\mid h_jh_{j+1}\mid\cdots\mid h_n]\\
&+(-1)^n[h_1\mid\cdots\mid h_{n-1}].
\end{aligned}
\]

Together with the augmentation map $\varepsilon\colon B_0(G)=\mathbb{C}[G]\to \mathbb{C}$,
this gives a free resolution
\[
    \cdots\to B_2(G)\to B_1(G)
    \to B_0(G)\to \mathbb{C}
    \to 0
\]
of the trivial left $\mathbb{C}[G]$-module $\mathbb{C}$.
If $M$ is a right $\mathbb{C}[G]$-module, we define the chain complex
\[
    C_\bullet(G;M) := M\otimes_{\mathbb{C}[G]}B_\bullet(G)
\]
and thus
\[
    \homol_n\bigl(C_\bullet(G;M)\bigr)=\homol_n(G;M).
\]
There is a natural identification
\[
    C_n(G;M)
    \cong
    \bigoplus_{(h_1,\ldots,h_n)\in G^n}
    M[h_1\mid\cdots\mid h_n].
\]
Under this identification, the boundary map is
\[
\begin{aligned}
\partial_n\bigl(m[h_1\mid\cdots\mid h_n]\bigr)
={}&mh_1[h_2\mid\cdots\mid h_n]\\
&+\sum_{j=1}^{n-1}(-1)^j
m[h_1\mid\cdots\mid h_jh_{j+1}\mid\cdots\mid h_n]\\
&+(-1)^n m[h_1\mid\cdots\mid h_{n-1}].
\end{aligned}
\]

The following lemma is completely standard (it is \cite[Exercise 3(a), Section V.5]{Brown1982}); we include a sketch proof for convenience.

\begin{lemma}\label{lem:direct-union-homology}
Suppose $G_0\leqslant G_1\leqslant \ldots $ are groups, let $G=\bigcup_{i\in \N}G_i$, and let $M$ be a fixed right $\C[G]$-module.
Then $\homol_n(G;M)=\varinjlim_{i\in\N} \homol_n(G_i;M)$, where $M$ is viewed as a $\C[G_i]$-module via restriction.
\end{lemma}
\begin{proof}[Sketch proof]
    By definition of the bar resolution, there are natural inclusions
    \[
        C_\bullet(G_0;M) \subseteq C_\bullet(G_1;M) \subseteq \cdots \subseteq C_\bullet(G;M),
    \]
    and $C_\bullet(G;M) = \bigcup_{i \in \N} C_\bullet(G_i;M)$. Then
    \[
        \homol_n(G;M) \cong \homol_n\bigl(C_\bullet(G;M)\bigr) \cong \varinjlim_{i\in\N} \homol_n\bigl(C_\bullet(G_i;M)\bigr) \cong \varinjlim_{i\in\N} \homol_n(G_i;M).
    \]
    The second isomorphism is the standard fact that the homology of a directed union of chain complexes is the directed limit of the corresponding homologies. To see this, observe that every cycle $z \in C_n(G;M)$ lies in $C_n(G_i;M)$ for some $i \in \N$, and moreover if $z = \partial_{n+1} c$ for some $c \in C_{n+1}(G;M)$, then again there is some $j \in \N$ (possibly greater than $i$) such that $c \in C_{n+1}(G_j;M)$. \qedhere
\end{proof}

\begin{rmk}
    The same result holds with directed unions of groups replaced by directed limits of groups, and the proof is similar. However, we will only need the directed union version of the statement.
\end{rmk}

Next, we recall some basics from the theory of $\ell^2$-invariants, and refer the reader to \cite{Luck2002} for more details. Let $G$ be a countable group. 
We consider the usual Hilbert space $\ell^2(G)$ with an orthonormal basis $\{\delta_g\mid g\in G\}$. The left regular representation is defined on basis elements by
\[
    \lambda(g)\delta_h=\delta_{gh} \qquad \text{for all} \ g,h \in G
\]
and extended to $\ell^2(G)$ by linearity. This defines an injective homomorphism $\lambda \colon \C[G] \to \mathcal B(\ell^2(G))$, where $\mathcal B(\ell^2(G))$ is the algebra of bounded operators on $\ell^2(G)$ acting on the right. The \emph{group von Neumann algebra} of $G$ is $\mathcal{N}(G)=\overline{\lambda(\C[G])}^{\textup{WOT}}=\lambda(G)''$, where the closure is taken with respect to the weak operator topology on $\mathcal B(\ell^2(G))$ and $\lambda(G)''$ is the bicommutant of $\lambda(G)$.

The \emph{ring of operators affiliated to $\mathcal{N}(G)$}, denoted by $\mathcal{U}(G)$, is the Ore localisation of the von Neumann algebra $\mathcal{N}(G)$ at its set of non-zero divisors. The fact that this is an Ore set is due to Berberian \cite{Berberian_vonNeumannOre}. We have the containments $\C[G]\subseteq \mathcal{N}(G)\subseteq\mathcal{U}(G)$ of rings.
L\"uck \cite[Section 8]{Luck2002} showed that there is a well-defined dimension function $\dim_{\mathcal U(G)}$ defined on all $\mathcal U(G)$-modules. The \emph{$\ell^2$-Betti numbers} of $G$ are then given by 
\[
    \beta_i^{(2)}(G) := \dim_{\mathcal U(G)} \homol_i(G;\mathcal U(G)).
\]
For a torsion-free group $G$, the \emph{Strong Atiyah Conjecture} is the statement that $\dim_{\mathcal U(G)}(\mathcal U(G) \otimes_{\C[G]} M) \in \Z$ for all finitely presented left $\C[G]$-modules $M$. 

The definition of $\dim_{\mathcal U(G)}$ is somewhat involved, but thankfully in our situation a conceptually simpler definition of $\ell^2$-Betti numbers is available.
Let $S$ be a unital ring. A subring $D \subseteq S$ is \emph{division closed} if $x \in D$ is invertible in $S$ implies $x^{-1} \in D$. Suppose $R \subseteq S$ is a (unital) subring. The \emph{division closure} of $R$ in $S$ is the smallest division closed subring of $S$ that contains $R$. We define the \emph{Linnell ring} $\D{G}$ of a countable group $G$ to be the division closure of $\C[G]$ in $\mathcal U(G)$. In general $\D{G}$ may not be a division ring (for instance when $G$ has torsion). Linnell \cite{Linnell1993} showed that when $G$ is torsion-free, the assertion that $\D{G}$ is a division ring is equivalent to $G$ satisfying the Strong Atiyah Conjecture. Also in \cite{Linnell1993}, Linnell showed that finite rank free groups satisfy the Strong Atiyah Conjecture, and therefore so does any countable locally free group (the Strong Atiyah Conjecture is easily seen to be stable under directed unions of torsion-free groups satisfying it). We thus have the following theorem, due to Linnell, which we will use in our computations.

\begin{thm}[{\cite{Linnell1993}}]\label{thm:Linnell-Atiyah}
Let $G$ be a countable locally free group. The following holds:
\begin{enumerate}
\item The group $G$ satisfies the Strong Atiyah Conjecture, and hence $\mathcal{D}_G$ is a division ring.
\item For each $n\in \N$, we have $\beta_n^{(2)}(G)= \dim_{\mathcal{D}_G}\homol_n(G;\mathcal{D}_G)$.
\end{enumerate}
\end{thm}

The dimension $\dim_{\mathcal D_G}$ is simply the rank of $\homol_n(G;\mathcal{D}_G)$ as a $\D{G}$-module (this is well defined, because all modules over division rings are free of unique rank). Computations using $\D{G}$ (instead of $\mathcal U(G)$) are easier since much of standard linear algebra can be carried out for modules over division rings, as in the proof of the following standard result.

\begin{lemma}\label{lem:L2-Betti-free-groups}
    Let $F_n$ be a finitely generated free group of rank $n>0$. If $\mathcal D$ is a division ring containing $\C[F_n]$ as a subring, then 
    \[
        \dim_{\mathcal D} \homol_i(G;\mathcal D) = \begin{cases}
            n-1 & \text{if} \ i = 1, \\
            0 & \text{otherwise.}
        \end{cases}
    \]
\end{lemma}
\begin{proof}[Sketch proof.]
    The trivial $\C[F_n]$-module $\C$ admits a free resolution of the form \[0 \to \C[F_n]^n \to \C[F_n] \to \C \to 0.\] After tensoring with $\mathcal D$, this becomes $0 \to \mathcal D^n \to \mathcal D \to 0$. The computation then immediately follows from the rank-nullity theorem (which holds for modules over division rings) and the fact that the map $\mathcal D^n \to \mathcal D$ is nonzero (and therefore surjective). \qedhere
\end{proof}

\section{The first \texorpdfstring{$\ell^2$}{ℓ²}-Betti number computation}

At the end of the introduction, we defined $\Gamma_n$ as the directed union $\bigcup_{m \in \N} F_{n+1}^{(m)}$, where the inclusion maps are denoted by $\phi_m \colon F_{n+1}^{(m)} \to F_{n+1}^{(m+1)}$. By Theorem \ref{thm:Linnell-Atiyah}, the group $\Gamma_n$ satisfies the Strong Atiyah Conjecture and therefore the Linnell ring $\mathcal D_{\Gamma_n}$ is a division ring.

\begin{lemma}\label{lem:L2-injectivity}
    The injection $\phi_m \colon F_{n+1}^{(m)} \to F_{n+1}^{(m+1)}$ induces an isomorphism
    \[
        \homol_1\left(F_{n+1}^{(m)}; \D{\Gamma_n}\right) \to \homol_1\left(F_{n+1}^{(m+1)}; \D{\Gamma_n}\right)
    \]
    of $\D{\Gamma_n}$-modules.
\end{lemma}

\begin{rmk}
    Lemma \ref{lem:L2-injectivity} shows that $F_{n+1}^{(i)}$ is \emph{$\ell^2$-independent} in $F_{n+1}^{(i+1)}$, to use the terminology introduced by Antol\'in and Jaikin-Zapirain in \cite{AntolinJaikin_HN}.
\end{rmk}

\begin{proof}[Proof (of Lemma \ref{lem:L2-injectivity})]
    For each $m \in \N$, let 
    \[
        0 \to \C[F_{n+1}^{(m)}] f_m \ \oplus \ \bigoplus_{i=1}^n \C[F_{n+1}^{(m)}] e_{m,i} \to \C[F_{n+1}^{(m)}] \to \C \to 0 \tag{$\dagger_m$}
    \]
    be a free resolution of the trivial $\C[F_{n+1}^{(m)}]$-module $\C$, where $\C[F_{n+1}^{(m)}] \to \C$ is the augmentation homomorphism, and 
    \[
        \C[F_{n+1}^{(m)}] f_m \ \oplus \ \bigoplus_{i=1}^n \C[F_{n+1}^{(m)}] e_{m,i} \to \C[F_{n+1}^{(m)}]
    \]
    is defined by
    \[
        f_m \mapsto a_m - 1 \qquad \text{and} \qquad e_{m,i} \mapsto s_{m,i} - 1.
    \]

    The homomorphism $\phi_m$ will induce a map from the resolution ($\dagger_m)$ to the resolution ($\dagger_{m+1})$ which lifts the identity map on the trivial module $\C$. Such a lift is well-defined up to chain homotopy and induces maps on homology with arbitrary right $\C[F_{n+1}^{(m+1)}]$-module coefficients. We now describe these maps.

    The map of resolutions is of the form
    \[
        \begin{tikzcd}
            0 \arrow[r] & \C[F_{n+1}^{(m)}] f_m \oplus \bigoplus_{i=1}^n \C[F_{n+1}^{(m)}] e_{m,i} \arrow[d, "(\phi_m)^{(1)}_*"]\arrow[r] & \C[F_{n+1}^{(m)}] \arrow[d, "(\phi_m)^{(0)}_*"]\arrow[r] & \C \arrow[d, "\mathrm{id}"] \\
            0 \arrow[r] & \C[F_{n+1}^{(m+1)}] f_{m+1} \oplus \bigoplus_{i=1}^n \C[F_{n+1}^{(m+1)}] e_{m+1,i} \arrow[r] & \C[F_{n+1}^{(m+1)}] \arrow[r] & \C
        \end{tikzcd}
    \]
    All of the maps in the diagram are $\C[F_{n+1}^{(m)}]$-module homomorphisms, where the modules on the bottom row are given a $\C[F_{n+1}^{(m)}]$-module structures via the map $\phi_m$. The vertical map $(\phi_m)^{(0)}_*$ is the usual map on group algebras induced by a group homomorphism; more precisely,
    \[
        (\phi_m)^{(0)}_*\left( \sum_{g \in F_{n+1}^{(m)}} \lambda_g g\right) \ = \  \sum_{g \in F_{n+1}^{(m)}} \lambda_g \phi_m(g).
    \]
    where $\lambda_g \in \C$ for all $g \in G$ and $\lambda_g = 0$ for all but finitely many $g \in G$. The more important homomorphism is $(\phi_m)^{(1)}_*$, which is defined on basis elements by 
    \[
        (\phi_m)^{(1)}_*(f_m) = f_{m+1}
    \]
    and
    \[
        (\phi_m)^{(1)}_*(e_{m,i}) = s_{m+1,i}^{-1}(a_{m+1}-1)e_{m+1,i} + s_{m+1,i}^{-1} f_{m+1}
    \]
    and extended to $\C[F_{n+1}^{(m)}] f_m \oplus \bigoplus_{i=1}^n \C[F_{n+1}^{(m)}] e_{m,i}$ by $\C[F_{n+1}^{(m)}]$-linearity. One easily checks that this definition makes the diagram commute.

    To obtain the induced map on $\D{\Gamma_n}$-homologies, we now tensor the diagram of resolutions with $\D{\Gamma_n}$ to obtain 
    \[
        \begin{tikzcd}
            0 \arrow[r] & \D{\Gamma_n} f_m \oplus \bigoplus_{i=1}^n \D{\Gamma_n} e_{m,i} \arrow[d, "(\phi_m)^{(1)}_*"]\arrow[r] & \D{\Gamma_n} \arrow[d, "(\phi_m)^{(0)}_*"]\arrow[r] & 0 \\
            0 \arrow[r] & \D{\Gamma_n} f_{m+1} \oplus \bigoplus_{i=1}^n \D{\Gamma_n} e_{m+1,i} \arrow[r] & \D{\Gamma_n} \arrow[r] & 0
        \end{tikzcd}
    \]
    where we have continued to use the notation $(\phi_m)^{(0)}_*$ and $(\phi_m)^{(1)}_*$ for the vertical maps, since they are obtained from the old vertical maps by extending coefficients to $\D{\Gamma_n}$ (and thus have the same definition on basis elements).

    It is clear from the definition of $(\phi_m)^{(1)}_*$ that the set 
    \[
        \{ (\phi_m)^{(1)}_*(f_m), \ (\phi_m)^{(1)}_*(e_{m,1}), \ \dots, \ (\phi_m)^{(1)}_*(e_{m,n}) \}
    \]
    is $\D{\Gamma_n}$-linearly independent. Indeed, the set is of the form
    \[
        \{ f_{m+1}, \ \alpha_1 e_{m+1,1} + \beta_1 f_{m+1}, \ \dots, \ \alpha_n e_{m+1,n} + \beta_n f_{m+1} \}
    \]
    for non-zero elements $\alpha_i, \beta_i \in \D{\Gamma_n}$. The elements $\{f_{m+1}, e_{m+1,1}, \dots, e_{m+1,n}\}$ are linearly independent by definition, and therefore so is the set on the previous line.
    
    The upshot of the previous paragraph is that $(\phi_m)^{(1)}_*$ is injective, and therefore an isomorphism since its domain and image are of the same $\D{\Gamma_n}$-dimension. Since the space of $1$-cycles over $\D{\Gamma_n}$ coincides with the degree one $\D{\Gamma_n}$-homology of the involved free groups, we obtain that $(\phi_m)^{(1)}_*$ restricts to an injection
    \[
        (\phi_m)^{(1)}_* \colon \homol_1\left(F_{n+1}^{(m)}; \D{\Gamma_n}\right) \hookrightarrow \homol_1\left(F_{n+1}^{(m+1)}; \D{\Gamma_n}\right).
    \]
    But again, $\homol_1(F_{n+1}^{(m)}; \D{\Gamma_n})$ and $\homol_1(F_{n+1}^{(m+1)}; \D{\Gamma_n})$ have the same $\D{\Gamma_n}$-dimension (by Lemma \ref{lem:L2-Betti-free-groups}) and therefore $(\phi_m)^{(1)}_*$ is an isomorphism. \qedhere
\end{proof}

\begin{lemma}\label{lem:L2-union}
    Let $G = \bigcup_{i \in \N} F^{(i)}$ be an increasing union of free groups, where $F^{(i)}$ is free of fixed rank $n+1$ for all $i \in \N$. If the inclusion induced map
    \[
        \homol_1(F^{(i)}; \D{G}) \to \homol_1(F^{(i+1)}; \D{G})
    \]
    is an isomorphism for all $i \in \N$, then $\beta_1^{(2)}(G) = n$ and $\beta_m^{(2)}(G) = 0$ for all $m \geqslant 2$.
\end{lemma}
\begin{proof}
    By Lemma \ref{lem:direct-union-homology}, we have
    \begin{equation}\label{eq:direct-limit}
        \homol_i(G; \D{G}) \cong \varinjlim_{i \in \N} \homol_i(F^{(i)}; \D{G}). \tag{$\ddagger$}
    \end{equation}
    This immediately implies that $\beta_m^{(2)}(G) = 0$ for all $m \geqslant 2$, since the $\D{G}$-homology of all the free groups $F^{(i)}$ vanishes above degree one (see Lemma \ref{lem:L2-Betti-free-groups}).

    The injectivity on $\D{G}$-homology implies that the inclusion induced map
    \[
        \homol_1(F^{(1)}; \D{G}) \to \homol_1(G; \D{G})
    \]
    is injective. Indeed, if $x \in \homol_1(F^{(1)}; \D{G})$ vanished in $\homol_1(G; \D{G})$, then $x$ would vanish in some $\homol_1(F^{(i)}; \D{G})$ for $i \in \N$ by \eqref{eq:direct-limit}. This implies that $x = 0$ in $\homol_1(F^{(1)}; \D{G})$. We conclude that 
    \[
        \beta_1^{(2)}(G) = \dim_{\D{G}} \homol_1(G;\D{G}) \geqslant \dim_{\D{G}} \homol_1(F^{(1)};\D{G}) = n.
    \]

    To prove the reverse inequality, let $x_1, \dots, x_{n+1} \in \homol_1(G;\D{G})$ be arbitrary elements. By \eqref{eq:direct-limit}, there is some $i \in \N$ such that $x_j \in \homol_1(F^{(i)};\D{G})$ for each $j=1, \dots, n+1$. Since
    \[
        \beta_1^{(2)}(F^{(i)}) = \dim_{\D{G}} \homol_1(F^{(i)};\D{G}) = n,
    \]
    we conclude that the elements $x_1, \dots, x_{n+1}$ are $\D{G}$-linearly dependent. Hence, 
    \[
        \beta_1^{(2)}(G) = \dim_{\D{G}} \homol_1(G;\D{G}) \leqslant n,
    \]
    as claimed. \qedhere
\end{proof}

\begin{proof}[Proof of Theorem \ref{thm:main-thm}]
The calculation of the $\ell^{2}$-Betti numbers follows from applying Lemmas \ref{lem:L2-injectivity} and \ref{lem:L2-union} to the group $\Gamma_n$ which is defined as the directed union $\bigcup_{m\in \N} F_{n+1}^{(m)}$. Finally, the fact that $n(\Gamma_n)=1$ follows immediately from construction.
\end{proof}

\bibliographystyle{alpha}
\bibliography{OsinThom}

@article{AntolinJaikin_HN,
  author   = {Antolín, Y. and Jaikin-Zapirain, A.},
  title    = {The {H}anna {N}eumann conjecture for surface groups},
  journal  = {Compositio Mathematica},
  fjournal = {Compositio Mathematica},
  volume   = {158},
  year     = {2022},
  number   = {9},
  pages    = {1850--1877},
  doi      = {10.1112/S0010437X22007709}
}

@article{Berberian_vonNeumannOre,
  author     = {Berberian, S. K.},
  title      = {The maximal ring of quotients of a finite von {N}eumann
                algebra},
  journal    = {Rocky Mountain J. Math.},
  fjournal   = {The Rocky Mountain Journal of Mathematics},
  volume     = {12},
  year       = {1982},
  number     = {1},
  pages      = {149--164},
  issn       = {0035-7596,1945-3795},
  mrclass    = {16A08 (16A30 46L10)},
  mrnumber   = {649748},
  mrreviewer = {David\ Handelman},
  doi        = {10.1216/RMJ-1982-12-1-149},
  url        = {https://doi.org/10.1216/RMJ-1982-12-1-149}
}

@incollection{BerrickHillman2008,
  author    = {Berrick, A. J. and Hillman, Jonathan A.},
  title     = {The {Whitehead} conjecture and {$L^2$}-{Betti} numbers},
  booktitle = {Guido's Book of Conjectures: A Gift to Guido Mislin on the
               Occasion of His Retirement from ETHZ, June 2006},
  editor    = {Chatterji, Indira},
  series    = {Monographies de L'Enseignement Math\'ematique},
  number    = {40},
  pages     = {35--37},
  publisher = {L'Enseignement Math\'ematique},
  address   = {Geneva},
  year      = {2008},
  isbn      = {978-2-940264-07-0}
}

@book{Brown1982,
  author    = {Brown, Kenneth S.},
  title     = {Cohomology of Groups},
  series    = {Graduate Texts in Mathematics},
  volume    = {87},
  publisher = {Springer-Verlag},
  address   = {New York},
  year      = {1982},
  doi       = {10.1007/978-1-4684-9327-6}
}

@article{Chen2026,
  author  = {Chen, Lvzhou},
  title   = {The {Kervaire} conjecture and the minimal complexity of surfaces},
  journal = {Transactions of the American Mathematical Society},
  volume  = {379},
  number  = {1},
  pages   = {587--626},
  year    = {2026},
  doi     = {10.1090/tran/9490},
  eprint  = {2302.09811},
  archiveprefix = {arXiv},
  primaryclass  = {math.GR}
}

@misc{ChenLodha2025,
  author        = {Chen, Lvzhou and Lodha, Yash},
  title         = {The {Wiegold} problem and free products of left-orderable groups},
  year          = {2025},
  eprint        = {2510.26073},
  archiveprefix = {arXiv},
  primaryclass  = {math.GR},
  note          = {arXiv:2510.26073v2}
}

@article{Howie1983,
  author  = {Howie, James},
  title   = {Some remarks on a problem of {J. H. C. Whitehead}},
  journal = {Topology},
  volume  = {22},
  number  = {4},
  year    = {1983},
  pages   = {475--485},
  doi     = {10.1016/0040-9383(83)90038-1}
}

@article{Klyachko1993,
  author  = {Klyachko, Anton A.},
  title   = {A funny property of sphere and equations over groups},
  journal = {Communications in Algebra},
  volume  = {21},
  number  = {7},
  pages   = {2555--2575},
  year    = {1993},
  doi     = {10.1080/00927879308824692}
}

@article{Levin1962,
  author  = {Levin, Frank},
  title   = {Solutions of equations over groups},
  journal = {Bulletin of the American Mathematical Society},
  volume  = {68},
  pages   = {603--604},
  year    = {1962},
  doi     = {10.1090/S0002-9904-1962-10868-4}
}

@article{Linnell1993,
  author  = {Linnell, Peter A.},
  title   = {Division rings and group von {Neumann} algebras},
  journal = {Forum Mathematicum},
  volume  = {5},
  number  = {6},
  pages   = {561--576},
  year    = {1993},
  doi     = {10.1515/form.1993.5.561},
  mrnumber = {1242889}
}

@book{Luck2002,
  author    = {L{\"u}ck, Wolfgang},
  title     = {{$L^2$}-Invariants: Theory and Applications to Geometry and
               {$K$}-Theory},
  series    = {Ergebnisse der Mathematik und ihrer Grenzgebiete. 3. Folge},
  volume    = {44},
  publisher = {Springer-Verlag},
  address   = {Berlin},
  year      = {2002},
  doi       = {10.1007/978-3-662-04687-6},
  mrnumber  = {1926649}
}

@article{OsinThom2013,
  author  = {Osin, Denis and Thom, Andreas},
  title   = {Normal generation and {$\ell^2$}-{Betti} numbers of groups},
  journal = {Mathematische Annalen},
  volume  = {355},
  number  = {4},
  pages   = {1331--1347},
  year    = {2013},
  doi     = {10.1007/s00208-012-0828-7},
  eprint  = {1108.2411},
  archiveprefix = {arXiv},
  primaryclass  = {math.GR},
  mrnumber = {3037017}
}

\end{document}